\documentclass{amsart}
\usepackage{amsmath,amscd,amsthm,verbatim,alltt,amsfonts,array}
\usepackage[english]{babel}
\usepackage{latexsym}
\usepackage{amssymb}
\usepackage{euscript}
\usepackage{nicefrac}
\usepackage{graphicx}
\def\Sum#1#2{\displaystyle \sum_{#1}^{#2}}
\def\DSum#1#2{\displaystyle \bigoplus_{#1}^{#2}}
\def\fracs#1#2{\displaystyle \frac{#1}{#2}}
\def\opn#1#2{\def#1{\operatorname{#2}}} 
\opn\chara{char} \opn\length{\ell} \opn\pd{pd} \opn\rk{rk}
\opn\projdim{proj\,dim} \opn\injdim{inj\,dim} \opn\rank{rank}
\opn\depth{depth} \opn\sdepth{sdepth} \opn\fdepth{fdepth}
\opn\grade{grade} \opn\height{height} \opn\embdim{emb\,dim}
\opn\codim{codim}  \opn\min{min} \opn\max{max}

\opn\Tr{Tr} \opn\bigrank{big\,rank}
\opn\superheight{superheight}\opn\lcm{lcm}
\opn\trdeg{tr\,deg}
\opn\reg{reg} \opn\lreg{lreg} \opn\ini{in} \opn\lpd{lpd}
\opn\size{size}

\opn\Spec{Spec} \opn\Supp{Supp} \opn\supp{supp} \opn\Sing{Sing}
\opn\Ass{Ass} \opn\Min{Min}
\opn\Ann{Ann} \opn\Rad{Rad} \opn\Soc{Soc}
\opn\Im{Im} \opn\Ker{Ker} \opn\Coker{Coker} \opn\Am{Am}
\opn\Hom{Hom} \opn\Tor{Tor} \opn\Ext{Ext} \opn\End{End}
\opn\Aut{Aut} \opn\id{id}  \opn\deg{deg}

\opn\nat{nat}
\opn\pff{pf}
\opn\Pf{Pf} \opn\GL{GL} \opn\SL{SL} \opn\mod{mod} \opn\ord{ord}
\opn\Gin{Gin} \opn\Hilb{Hilb}
\let\iso=\cong

\let\to=\rightarrow
\let\xto=\xrightarrow

\def\Implies{\ifmmode\Longrightarrow \else
        \unskip${}\Longrightarrow{}$\ignorespaces\fi}
\def\implies{\ifmmode\Rightarrow \else
        \unskip${}\Rightarrow{}$\ignorespaces\fi}
\def\iff{\ifmmode\Longleftrightarrow \else
        \unskip${}\Longleftrightarrow{}$\ignorespaces\fi}

\let\:=\colon
\newtheorem{theorem}{Theorem}
\theoremstyle{plain}

\newtheorem{corollary}{Corollary}

\newtheorem{definition}{Definition}
\newtheorem{example}{Example}

\newtheorem{lemma}{Lemma}

\newtheorem{proposition}{Proposition}
\newtheorem{remark}{Remark}

\numberwithin{equation}{section}
\let\epsilon\varepsilon
\let\phi=\varphi
\let\kappa=\varkappa

\begin{document}

\title{Depth and Stanley Depth of the  Canonical Form of a factor of monomial ideals}
\author{Adrian Popescu}
\date{}
\pagestyle{myheadings}
\thanks{The  support from the Department of Mathematics of the University of Kaiserslautern is gratefully acknowledged.}

\address{Adrian Popescu, Department of Mathematics, University of Kaiserslautern, Erwin-Schr\"odinger-Str., 67663 Kaiserslautern, Germany}
\email{popescu@mathematik.uni-kl.de}

\maketitle

\begin{abstract}
We introduce a so called canonical form of a factor of two monomial ideals. The depth and the Stanley depth of such a factor is invariant under taking the canonical form. This can be seen using a result of Okazaki and Yanagawa [7]. In the case of depth we present in this paper a different proof. It follows easily that the Stanley Conjecture holds for the factor if and only if it holds for its canonical form. In particular, we construct an algorithm which simplifies the depth computation and using the canonical form we massively reduce the run time for the sdepth computation.
\end{abstract}

\thispagestyle{empty}

\section{Introduction}
\vskip 1cm

 Let $K$ be a field and $S=K[x_1,\ldots,x_n]$ be the polynomial ring over $K$
 in $n$ variables. A Stanley decomposition of a graded $S-$module $M$ is a finite family $$\mathcal{D} = (S_i, u_i)_{i \in I}$$ in which $u_i$ are homogeneous elements of $M$ and $S_i$ are graded $K-$algebra retract if $S$ for all $i  \in I$ such that $S_i \cap \Ann (u_i) = 0$ and $$M = \DSum{i \in I}{} S_iu_i$$ as a graded $K-$vector space.  The Stanley depth of $\mathcal{D}$, denoted by $\sdepth (\mathcal{D})$, is the depth of the $S-$module $\DSum{i \in I}{} S_iu_i$. The Stanley depth of $M$ is defined as   $$ \sdepth\ (M) :=\max\{\sdepth \ ({\mathcal D})\ |\  {\mathcal D}\; \text{is a
Stanley decomposition of}\;  I \}.$$

Another definition of sdepth using partitions is given in \cite{HVZ}.

Stanley's Conjecture \cite{S} states that the Stanley depth \\ $\sdepth (M)$ is $\geq \depth\ (M)$.

Let $J \subsetneq I \subset S$ be two monomial ideals in $S$. In \cite{IKF}, Ichim et. al.  studied the sdepth and depth of the factor $\nicefrac{I}{J}$ under polarization and  reduced the Stanley's Conjecture to the case when the ideals are monomial squarefree. This is possible the best result from the last years concerning Stanley's depth. It is worth to mention that this result is not very useful for computing sdepth since it introduces a lot of new variables.   In the squarefree case there are not many known results about the Stanley conjecture (see for example \cite{PP}).

Another result of \cite{IKF} which helps in the sdepth computation is the following proposition, which extends \cite[Lemma 1.1]{Ci}, \cite[Lemma 2.1]{IQ}.

\begin{proposition}\label{main} \cite[Proposition 5.1]{IKF} Let $k\in {\mathbb N}$ and $I''$, $J''$ be the monomial ideals obtained from $I$, $J$ in the following way: Each generator whose degree in $x_n$ is at least $k$ is multiplied by $x_n$ and all other generators are taken unchanged. Then
$\sdepth_S\nicefrac{I}{J} = \sdepth_S\nicefrac{I''}{J''}$.
\end{proposition}

Inspired by this proposition we introduce a canonical form of a factor $\nicefrac{I}{J}$ of monomial ideals (see Definition \ref{def: canonical quot}) and we prove easily that sdepth is invariant  under taking the canonical form (see Theorem \ref{prop:sdepth}). This  leads us to the idea to study also the depth case (see Theorem \ref{prop:depth}). Theorem \ref{m} says that Stanley's Conjecture holds for a factor of monomial ideals if and only if it holds for its canonical form. As a side result, in the depth (respectively sdepth) computation algorithm for $\nicefrac{I}{J}$, one can first compute the canonical form and use the algorithm on this new much more simpler module (see the Appendix).

In Example \ref{ex: timings} we conclude that the $\depth$ and $\sdepth$ algorithms are faster when considering the canonical form: using $\textsc{CoCoA}$\cite{Co}, $\textsc{Singular}$\cite{Sing} and Rinaldo's $\sdepth$ computation algorithm \cite{Rina} we see a small decrease in the $\depth$ case timing, but in the $\sdepth$ case the run time is massively reduced. We hope that our algorithm together with the one from \cite{AP} will be used very often in problems concerning monomial ideals.

We owe thanks to Y.-H. Shen who noticed our results in a previous arXiv version and showed us the papers of Okazaki and Yanagawa \cite{OY} and \cite{Y}, because they are strongly connected with our topic. Indeed Proposition \ref{main} and Corollary \ref{prop:5.1 for depth } follow from \cite[Theorem 5.2]{OY} (see also \cite[Section 2,3]{OY}). However, our proofs of  Lemma \ref{lemma:depth is constant} and Corollary \ref{prop:5.1 for depth } are completely different from those appeared in the quoted papers and we  keep them for the sake of our completeness.
\vskip 1cm
\section{The canonical form of a factor of monomial ideals}
\vskip 1cm

Let $R=K[x_1,\ldots,x_{n-1}]$ be the polynomial $K$-algebra over a field $K$ and $S := R[x_n]$. Consider $J \subsetneq I\subset R$ two monomial ideals and denote by $G(I)$, respectively $G(J)$, the minimal (monomial) system of generators of $I$, respectively $J$.

\begin{definition} \label{def: canonical}
 {\em
 The power $x_n^r$ {\em enters in a monomial} $u$ if  $x_n^r|u$ but $x_n^{r+1}\nmid u$.

We say that $I$ is {\em of type} $(k_1,\ldots,k_s)$ {\em with respect to} $x_n$ if $x_n^{k_i}$ are all the powers of $x_n$ which enter in a monomial of $G(I)$ for $i\in [s]$ and $1\leq k_1<\ldots<k_s$.

$I$ is {\em in the canonical form with respect to} $x_n$ if $I$ is of type $(1,\ldots,s)$ for some $s\in {\mathbb  N}$.

We simply say that $I$ is {\em the canonical form} if it is in the canonical form with respect to all variables $x_1, \ldots, x_n$.

}
\end{definition}

\begin{remark} \label{rem:canonical form}{\em
Suppose that  $I$ is  of type $(k_1,\ldots,k_s)$ with respect to $x_n$. It is easy to get the {\em canonical form} $I'$ of $I$ {\em with respect to} $x_n$: replace $x_n^{k_i}$ by $x_n^i$ whenever $x_n^{k_i}$ enters in a generators of $G(I)$. Applying by recurrence this  procedure for other variables we get the {\em canonical form} of $I$, that is with respect to all variables. Note that a squarefree monomial ideal is of type $(1)$ with respect to each $x_i$ and it is in the canonical form with respect to $x_i$, so in this case $I'=I$.  }
\end{remark}

\begin{definition} \label{def: canonical quot} {\em
Let $J \subsetneq I \subset S$ two monomial ideals.
We say that $\nicefrac{I}{J}$ is {\em of type} $(k_1, \ldots, k_s)$ {\em with respect to } $x_n$ if $x_n^{k_i}$ are all the powers of $x_n$ which enter in a monomial of $G(I) \cup G(J)$ for $i \in [s]$ and $1 \leq k_1 < \ldots < k_s $.
All the terminology presented in Definition \ref{def: canonical} will extend automatically to the factor case. Thus we may speak about
the {\em canonical form}  $\overline{\nicefrac{I}{J}}$ of $\nicefrac{I}{J}$.
}
\end{definition}

\begin{remark} \label{rem:canonical quot}
{\em
In order to compute the canonical form with respect to $x_n$ of the $(k_1, \ldots, k_s)-$type factor $\nicefrac{I}{J}$, one will replace $x_n^{k_i}$ by $x_n^i$ whenever $x_n^{k_i}$ enters a generator of $G(I) \cup G(J)$.
}
\end{remark}

\begin{example}\label{ex:canonical}
{\em We present some examples where we compute the canonical form of a monomial ideal, respectively a factor of two monomial ideals. \\
\begin{enumerate}
\item Consider $S = \mathbb Q [x,y]$ and the monomial ideal $I = (x^4,x^3y^7)$. Then the canonical form of $I$ is $I' = (x^2, xy)$.

\item Consider $S = \mathbb Q [x,y,z]$, $I = (x^{10}y^5, x^4yz^7,z^7y^3)$ and \\ $J = (x^{10}y^{20}z^2, x^3y^4z^{13}, x^9y^2z^7)$.

The canonical form of $\nicefrac{I}{J}$ is $\overline{\nicefrac{I}{J}} = \fracs{(x^4y^5, x^2yz^2, y^3z^2)}{(x^4y^6z, xy^4z^3, x^3y^2z^2)}$.
\end{enumerate}}
\end{example}

The canonical form of a factor of monomial ideals $\nicefrac{I}{J}$ is not usually the factor of the canonical forms of $I$ and $J$ as shows the following example.

\begin{example}{\em
Let $S = \mathbb Q [x, y]$, $I = (x^4, y^{10}, x^2y^7)$ be
 and $J = (x^{20}, y^{30})$.
The canonical form of $I$ is $I' = (x^2, y^2, xy)$ and the canonical form of $J$ is $J' = (x,y)$. Then $J'\not\subset I'$.
But the canonical form of the factor $\nicefrac{I}{J}$ is $\overline{\nicefrac{I}{J}} = \fracs{(x^2,y^2,xy)}{(x^3, y^3)}$.
}
\end{example}

Using Proposition \ref{main}, we see that the Stanley depth of a monomial ideal does not change when considering its canonical form.

\begin{theorem}\label{prop:sdepth}
Let $I$, $J$ be monomial ideals in $S$ and $\overline{\nicefrac{I}{J}}$ the canonical form of $\nicefrac{I}{J}$. Then  $$\sdepth_S \nicefrac{I}{J}=\sdepth_S \overline{\nicefrac{I}{J}}.$$
\end{theorem}

The proof goes applying  inductively the following lemma.

\begin{lemma}\label{lemma: sdepth}
Suppose that  $\nicefrac{I}{J}$ is  of type $(k_1,\ldots,k_s)$ with respect to $x_n$ and $k_j+1<k_{j+1}$ for some $0\leq j<s$ (we set $k_0=0$).
Let $G(I')$ (resp. $G(J')$) be the set of monomials  obtained from $G(I)$  (resp. $G(J)$) by substituting $x_n^{k_i} $ by $x_n^{k_i-1}$ for $i>j$ whenever $x_n^{k_i}$ enters in a monomial of $G(I)$ (resp. $G(J)$). Let $I'$ and $J'$ be the ideals generated by $G(I')$ and $G(J')$. Then $$\sdepth_S \nicefrac{I}{J}=\sdepth_S \nicefrac{I'}{J'}.$$
\end{lemma}

The proof of Lemma \ref{lemma: sdepth} follows from the proof of \cite[Proposition 5.1]{IKF} (see here Proposition \ref{main}).

Next we focus on the $\depth \nicefrac{I}{J}$ and $\depth \overline{\nicefrac{I}{J}}$. The idea of the proof of the following lemma is taken from \cite[Section 2]{P}.

\begin{lemma}\label{lemma:depth is constant}
Let $I_0 \subset I_1  \subset \ldots \subset I_e \subset R$, $J\subset S$, $U_0 \subset U_1 \subset \ldots \subset U_e \subset R$, $V \subset S$ be some graded ideals of $S$, respectively $R$, such that $U_i \subset I_i$ for $0 \leq i \leq e$, $I_e \subset J$, $V \subset J$ and $U_e \subset V$. Consider $T_k = \Sum{i=0}{e}x_n^i I_i S + x_n^k J$ and $W_k = \Sum{i=0}{e}x_n^iU_iS + x_n^kV$ for $k > e$. Then $\depth_S \fracs{T_k}{W_k}$ is constant for all $k>e$.
\end{lemma}
\begin{proof}

Consider the following linear subspaces of $S$: $I := \Sum{i=0}{e}x_n^i I_i$ and $U := \Sum{i=0}{e}x_n^i U_i$. Note that $I$ and $U$ are not ideals in $S$.

If $I = U$, then the claim follows easily from the next chain of isomorphisms
$\fracs{T_k}{W_k} \iso \fracs{x_n^kJ}{x_n^kJ \cap (I+x_n^kV)S} \iso \fracs{x_n^kJ}{x_n^k(I+V)S} \iso \fracs{J}{(I+V)S}$ for all $k>e$, and hence $\depth_S \fracs{T_k}{W_k}$ is constant for all $k>e$.

Assume now that $I \neq U$ and consider the following exact sequence
$$0 \to \fracs{J}{V} \xto{\cdot x_n^k} \fracs{T_k}{W_k} \to \fracs{T_k}{W_k + x_n^kJ} \to 0,$$ where the last term we denote by $H_k$. Note that $H_k \iso  \fracs{IS}{IS \cap (U+x_n^kJ)S}$ and  $IS \cap (U + x_n^kJ)S = US + x_n^kIS$. Since $x_n^kH_k=0$, $H_k$ is a $\nicefrac{S}{(x_n^k)}-$module. Then $\depth_S H_k = \depth_{\nicefrac{S}{(x_n^k)}}H_k=\depth_R H_k$ because the graded maximal ideal $m$ of $R$ generates a zero dimensional ideal in $\nicefrac{S}{(x_n^k)}$. But $H_k$ over $R$ is isomorphic with $\fracs{\oplus_{i=0}^{k-1} I_i}{\oplus_{i=0}^{k-1} U_i} \iso \bigoplus_{i=0}^{k-1} \fracs{I_i}{U_i}$, where $I_i=I_e$ and $U_i=U_e$ for $e<i<k$. It follows that $t:=\depth_S H_k = \min_i\left\{\depth_{R} \fracs{I_i}{U_i}\right\}$.

If $\depth_S \fracs{J}{V} = 0$, then the Depth Lemma gives us $\depth_S \fracs{T_k}{W_k} = t = 0$ for all $k>e$ and hence we are done. Therefore we may suppose that $\depth_S \fracs{J}{V} > 0$. Note that $t>0$ implies $\depth_S \fracs{T_k}{W_k} > 0$ by the Depth Lemma since otherwise $\depth_S \fracs{T_k}{W_k} = \depth_{S} \fracs{J}{V} = 0$, which is false. Next we will split the proof in two cases.

$\circ$ Case $t = 0$.

Let ${\mathcal F}=\big\{i\in \{0,\ldots,e\}\ \big|\ \depth_R\nicefrac{I_i}{U_i}=0\big\}$ and $L_i\subset I_i$ be the graded ideal containing $U_i$ such that $\nicefrac{L_i}{U_i} \iso H_m^0(\nicefrac{I_{i}}{U_{i}})$.

If $i\in {\mathcal F}$ and there exists $u\in ( L\cap V)\setminus U_i$ then $(m^s,x_n^k)x_n^i u\subset W_k$ for some $s\in {\mathbb N}$, that is $\depth_S \fracs{T_k}{W_k} = 0$ for all $k > e$.

Now consider the case when $L_i\cap V=U_i$ for all $i\in {\mathcal F}$. If $i\in {\mathcal F}$ then  note that $L_i\subset L_j$  for $i< j\leq e$.
Set $V'=V+L_eS$, $U'=U+ \Sum{i\in {\mathcal F}}{}x_n^i L_i$ and $W'_k := U'S+x_n^kV'=U'S+x_n^kV$ because $x_n^kL_eS\subset U'S$. Consider the following exact sequence $$0 \to \fracs{W'_k}{W_k} \to \fracs{T_k}{W_k} \to \fracs{T_k}{W_k'} \to 0.$$ For the last term we have $H_m^0(\nicefrac{I_j}{U'_j})=0$, $0\leq j\leq e$ and so the new $t>0$, which is our next case. Thus we get $\depth_S \fracs{T_k}{W'_k}>0$ is constant for $k>e$. The first term is isomorphic  to $\fracs{U'S}{U'S\cap W_k}$. But $U'S\cap W_k=US+(U'S\cap x_n^kV)$ since $US\subset U'S$. Since $U'S\cap (x_n^kS)=x_n^k (U_e+L_e)S$ and $U_e\subset V$ it follows that $U'S\cap x_n^kV=x_n^kUS+(x_n^kL_eS\cap x_n^kVS)=x_n^kUS$.
Consequently, the first term from the above exact sequence is isomorphic with $\fracs{U'S}{US}$. Note that the annihilator  of the element induced by some $u\in L_e\setminus V$ in $\nicefrac{U'S}{US}$ contains a power of $m$ and so $\depth_S \fracs{U'S}{US}\leq 1$. The inequality is equality since $x_n$  is regular on $\nicefrac{U'S}{US}$. By the Depth Lemma we get $\depth_S \fracs{T_k}{W_k}=1$ for all $k>e$.

$\circ$ Case $t>0$.

If $\depth_{R} \fracs{J}{V}\leq t= \depth_S H_k$ then the Depth Lemma gives us again the claim, i.e. $\depth_S \fracs{T_k}{W_k} = \depth_S \fracs{J}{V}$ for all $k>e$.

Assume that $\depth_{S} \fracs{J}{V} >t$. Apply induction on $t$, the initial step $t = 0$ being done in the first case. Suppose that $t>0$. Then $\depth_S \fracs{J}{V}>t>0$ implies that $\depth_S \fracs{J}{V}\geq 2$ and so
 we may find a homogeneous polynomial $f \in m$ that is regular on $\fracs{J}{V}$. Moreover we may find $f$ to be regular also   on all $\fracs{I_i}{U_i}$, $i \leq e$. Then $f$ is regular on $\fracs{T_k}{W_k}$. Set $V'' := V + f J$ and $U''_i := U_i + f I_i$ for all $i \leq e$ and set $W''_k := \Sum{i=0}{e}x_n^iU''_iS + x_n^kV''$. By Nakayama's Lemma we get $U'' \neq U$, and therefore $\depth_{R} \fracs{I}{U''} = t-1$ and by induction hypothesis it results that $\depth_S \fracs{T_k}{W_k} = 1 + \depth_S \fracs{T_k}{W''_k} = $ constant for all $k > e$.

 Finally, note that we may pass from the first case to the second one and conversely. In this way $U$ increases at each step. By Noetherianity at last we may arrive in finite steps to the case $I=U$, which was solved at the beginning.
 \hfill\
\end{proof}

The next corollary is in fact \cite[Proposition 5.1]{IKF} (see Proposition \ref{main}) for depth. It follows easily from Lemma \ref{lemma:depth is constant} but also from \cite[Proposition 5.2]{OY} (see also \cite[Sections 2, 3]{Y}.

\begin{corollary}\label{prop:5.1 for depth }
Let $e \in \mathbb N$, $I$ and $J$ monomial ideals in $S := K[x_1, \ldots, x_n]$. Consider $I'$ and $J'$ be the monomial ideals obtained from $I$ and $J$ in the following way: each generator whose degree in $x_n$ $\geq e$ is multiplied by $x_n$ and all the other generators are left unchanged. Then
$$\depth_S \nicefrac{I}{J} = \depth_S \nicefrac{I'}{J'}.$$
\end{corollary}

This leads us to the equivalent result of Theorem \ref{prop:sdepth} for depth.

\begin{theorem} \label{prop:depth}
Let $I$ and $J$ be two monomial ideals in $S$ and $\overline{\nicefrac{I}{J}}$ the canonical form of $\nicefrac{I}{J}$. Then $$\depth_S \nicefrac{I}{J} = \depth_S \overline{\nicefrac{I}{J}}.$$
\end{theorem}
\begin{proof}
Assume that $\nicefrac{I}{J}$ is of type $(k_1,\ldots,k_s)$ with respect to $x_n$ and obviously $\overline{\nicefrac{I}{J}}$ is of type $(1,2,\ldots,s)$ with respect to $x_n$. Starting with $\overline{\nicefrac{I}{J}}$, we apply Corollary \ref{prop:5.1 for depth } till we obtain an $\nicefrac{I'_1}{J'_1}$ of type $(k_1,k_1+1,\ldots,k_1+s-1)$ having the same depth as $\overline{\nicefrac{I}{J}}$. We repeat the process until we get $\nicefrac{I'_s}{J'_s}$ of type $(k_1, k_2, \ldots, k_s)$ with respect to $x_n$ with the unchanged depth. Now we iterate and take the next variable. At the very end the claim will follow.\hfill \
\end{proof}

Theorem \ref{prop:sdepth} and Theorem \ref{prop:depth} give us the following theorem

\begin{theorem}\label{m}
The Stanley conjecture holds for a factor of monomial ideals $\nicefrac{I}{J}$ if and only if it holds for its canonical form $\overline{\nicefrac{I}{J}}$.
\end{theorem}

Using  Theorem \ref{prop:depth}, instead of computing the $\depth$ or the $\sdepth$ of $\nicefrac{I}{J}$, $J \subsetneq I \subset S$, we can compute it for the simpler module $\overline{\nicefrac{I}{J}}$.

\begin{example} \label{ex: timings}
{\em We present the different timings for the depth and sdepth computation algorithms with and without extracting the canonical form. $\textsc{Singular}$\cite{Sing} was used in the depth computations while $\textsc{CoCoA}$ \cite{Co} and Rinaldo's paper\cite{Rina} were used for the Stanley depth computation.
\begin{enumerate}
\item Consider the ideals from Example \ref{ex:canonical}(2).

Timing for $\sdepth \nicefrac{I}{J}$  computation: 22s.

Timing for $\sdepth \overline{\nicefrac{I}{J}}$ computation: 74 ms.

\item Consider $R = \mathbb Q[x,y,z]$ and $I = (x^{100}yz,x^{50}yz^{50},x^{50}y^{50}z)$. Then the canonical form is $I' = (x^2yz,xyz^2,xy^2z)$.

Timing for $\sdepth I$  computation: 13m 3s.

Timing for $\sdepth I'$ computation: 21 ms.

Notice that the difference in timings is very large. Therefore using the canonical form in the $\sdepth$ computation is a very important optimization step. On the other side, the $\depth$ computation is immediate in both cases. In the last example, the timing difference can be seen.

\item Consider $R = \mathbb{Q}[x,y,z,t,v,a_1,\ldots,a_5]$,\\ $I = (v^4x^{12}z^{73},v^{87}t^{21}y^{13},x^{43}y^{18}z^{72}t^{28},vxy,vyz,vzt,vtx,a_1^{7000}, a_2^{413};)$, \\$J = (v^5x^{13}z^{74},v^{88}t^{22}y^{14},x^{44}y^{19}z^{73}t^{29},v^2x^2y^2,v^2y^2z^2,v^2z^2t^2,v^2t^2x^2)$.

Timing for $\depth \nicefrac{I}{J}$ computation: 16m 11s.

Timing for $\depth \overline{\nicefrac{I}{J}}$ computation: 11m.

\end{enumerate}
}
\end{example}

\vskip 1cm
\section{Appendix}
\vskip 1cm
We sketch the simple idea of the algorithm which computes the canonical form of a monomial ideal $I$. This can easily be extended to compute the canonical form of $\nicefrac{I}{J}$ by simple applying it for $G(I) \cup G(J)$ and afterwards extracting the generators corresponding to $I$ and $J$. This was used in Example \ref{ex: timings}.

The algorithm is based on Remark \ref{rem:canonical quot}: for each variable $x_i$ we build the list \verb"gp" in which we save the pair $(g,p)$, were $p$ is chosen such that $x_i^p$ enters the $g-$generator of the monomial ideal $I$. This list will be sorted by the powers $p$ as in the following example

\begin{example}
{\em Consider the ideal $I := (x^{13}, x^{4}y^{7}, y^7z^{10}) \subset \mathbb{Q} [x, y, z]$. Then for each variable we will obtain a different \verb"gp" as shown below:
\begin{itemize}

\item[$\circ$] For the first variable $x$, \verb"gp" is equal to \begin{tabular}{ |c | c | c |c| }
  \hline
  2 & 4 & 1 & 13 \\
  \hline
\end{tabular}. Therefore $I$ is of type $(4,13)$ with respect to $x$. Hence, in order to obtain the canonical form with respect to $x$, one has to divide the second generator by $x^{4-1} = x^3$ and the first generator by $x^{13-2} = x^{11}$. After these computation we will get $I_1 = (x^2, xy^7, y^7z^{10})$. Note that $I_1$ is in the canonical form w.r.t. $x$.

\item[$\circ$] For the second variable $y$, \verb"gp" is equal to \begin{tabular}{ |c | c | c |c| }
  \hline
  3 & 7 & 2 & 7 \\
  \hline
\end{tabular}. Similar as above, one has to divide the second and the third generator by $y^6$, and hence it results $I_2 = (x^2, xy, yz^{10})$. Again, $I_2$ is in the canonical form w.r.t. $y$ and $x$.

\item[$\circ$] For the last variable $z$, \verb"gp" is equal to   \begin{tabular}{ |c | c | }
  \hline
  3 & 10 \\
  \hline
\end{tabular}. We divide the third generator of $I_2$ by $z^9$ and we get our final result $I' = (x^2, xy, yz)$., which is in the canonical form with respect to all variables.
\end{itemize}
}
\end{example}

Based on the above idea, we construct two procedures: \verb"putIn" and \verb"canonical" $-$ the first one constructing the list \verb"gp", and the second one computing the canonical form of a monomial ideal. The proof of correctness and termination is trivial. The procedures were written in the $\textsc{Singular}$ language.
\begin{verbatim}
proc putIn(intvec v, int power, int nrgen)
{
    if(size(v) == 1)
    {
        v[1] = nrgen;
        v[2] = power;
        return(v);
    }
    int i,j;
    if(power <= v[2])
    {
        for(j = size(v)+2; j >=3; j--)
        {
            v[j] = v[j-2];
        }
        v[1] = nrgen;
        v[2] = power;
        return(v);
    }
    if(power >= v[size(v)])
    {
        v[size(v)+1] = nrgen;
        v[size(v)+1] = power;
        return(v);
    }
    for(j = size(v) + 2; (j>=4) && (power < v[j-2]); j = j-2)
    {
        v[j] = v[j-2];
        v[j-1] = v[j-3];
    }
    v[j] = power;
    v[j-1] = nrgen;
    return(v);
}

proc canonical(ideal I)
{
    int i,j,k;
    intvec gp;
    ideal m;
    intvec v;
    v = 0:nvars(basering);
    for(i = 1; i<=nvars(basering); i++)
    {
        gp = 0;
        v[i] = 1;
        for(j = 1; j<=size(I); j++)
        {
            if(deg(I[j],v) >= 1)
            {
                gp = putIn(gp,deg(I[j],v),j);
            }
        }
        k = 0;
        if(size(gp) == 2)
        {
            I[gp[1]] = I[gp[1]]/(var(i)^(gp[2]-1));
        }
        else
        {
            for(j = 1; j<=size(gp)-2;)
            {
                k++;
                I[gp[j]] = I[gp[j]]/(var(i)^(gp[j+1]-k));
                j = j+2;
                while((j<=size(gp)-2) && (gp[j-1] == gp[j+1]) )
                {
                    I[gp[j]] = I[gp[j]]/(var(i)^(gp[j+1]-k));
                    j = j + 2;
                }
            }
            if(j == size(gp)-1)
            {
                if(gp[j-1] == gp[j+1])
                {
                    I[gp[j]] = I[gp[j]]/(var(i)^(gp[j+1]-k));
                }
                else
                {
                    k++;
                    I[gp[j]] = I[gp[j]]/(var(i)^(gp[j+1]-k));
                }
            }
        }
        v[i] = 0;
    }
    return(I);
}
\end{verbatim}

\end{document}